\documentclass[11pt]{amsart}

\usepackage[english]{babel}
\usepackage[utf8]{inputenc}
\usepackage[T1]{fontenc}

\usepackage{amsmath}
\usepackage{amssymb}
\usepackage{amsfonts}
\usepackage{amssymb}
\usepackage{amsthm}
\usepackage{amscd}

\usepackage[hidelinks]{hyperref}
\usepackage{geometry}

\usepackage{marvosym}

\newcommand{\mengensymbol}[1]{\ensuremath{\mathbb{#1}}}%
\newcommand{\R}{\mengensymbol{R}}%

\newtheorem{definition}{Definition}[section]
\newtheorem{lemma}[definition]{Lemma}
\newtheorem{theorem}[definition]{Theorem}
\newtheorem{proposition}[definition]{Proposition}
\newtheorem{corollary}[definition]{Corollary}

\DeclareMathOperator{\Rm}{Rm}
\DeclareMathOperator{\Ric}{Ric}

\DeclareMathOperator{\divA}{div}
\DeclareMathOperator{\trace}{tr}
\DeclareMathOperator{\Mat}{Mat}
\DeclareMathOperator{\CurvatureDimension}{CD}
\DeclareMathOperator{\seg}{seg}

\title{Sobolev inequalities on manifolds
with nonnegative Bakry-Émery Ricci curvature}

\author{Florian Johne}

\address{Department of Mathematics \\ Columbia University \\ New York NY 10027}

\begin{document}
 
  \begin{abstract}
 In this note we extend a recent result of S.~Brendle \cite{Brendle-NonnegativeCurvature} to Riemannian manifolds with densities
 and nonnegative Bakry--Émery Ricci curvature.
 \end{abstract}
 
 \maketitle

 \section{Introduction}
 
 Suppose $(M,g)$ is a smooth complete noncompact Riemannian manifold of dimension $\dim M = m$.
 Given a positive smooth function $w$ (called the density) and $\alpha > 0$ we may consider the 
 Bakry-Émery Ricci curvature of the Riemannian metric measure space $(M,g,w \, d\mu)$ given by
 \[
  \Ric_w^{\alpha}
:= \Ric - D^2 (\log w) - \frac{1}{\alpha} D\log w \otimes D \log w,
 \]
where $\Ric$ denotes the Ricci curvature of the Riemannian metric $g$, $d\mu$ denotes the volume form associated of the Riemannian metric $g$, $Dv$ denotes the differential of a function $v$ and $D^2 v$ denotes the Hessian of a function $v$.
If the Bakry-Émery Ricci curvature is nonnegative, the triple $(M,g,w \, d\mu)$ is a $\CurvatureDimension(0,m+\alpha)$ space.

We define the $\alpha$-asymptotic volume ratio $\mathcal{V}_{\alpha}$ of $(M,g,w \, d\mu)$ by
\begin{align*}
 \mathcal{V}_{\alpha}
 =
 \lim_{r \rightarrow \infty}
 \frac{1}{ r^{m+\alpha}} \int_{B_r(q)} w, \;  \text{where} \; B_{r}(q)
 = \left\{ p \in M | d_g(p,q) < r \right\},
\end{align*}
 the function $d_g$ is the distance function induced by the Riemannian metric $g$
 and $q \in M$ is some given point.
 The analogue of the Bishop--Gromov volume comparison for manifolds
 with densities
 due to K.-T.~Sturm \cite[Theorem 2.3]{Sturm} implies
 that the limit exists and is independent of $q \in M$.
 In the appendix we provide a direct proof for this result in the smooth setting.
 
 We extend a result (Theorem 1.1 in \cite{Brendle-NonnegativeCurvature}) of S.~Brendle
 to our setting:
 
 \begin{theorem} \ \\
 Let $(M,g)$ be a smooth complete noncompact Riemannian manifold of dimension $\dim M = m$
 with smooth positive density $w$ and nonnegative Bakry-Emery Ricci curvature $\Ric_{\alpha}^w$.
Suppose $K \subset M$ is a compact domain with smooth boundary $\partial K$
and let $f$ be a smooth positive function on $K$.
 Then we have the estimate
  \begin{equation}
   \int_{K} w |D f| + \int_{\partial K} w f \geq 
   (m+\alpha) \, \mathcal{V}_{\alpha}^{\frac{1}{m+\alpha}}
   \left( \int_K w f^{\frac{m+\alpha}{m+\alpha-1}} \right)^{\frac{m+\alpha-1}{m+\alpha}},
  \end{equation}
  where $\mathcal{V}_{\alpha}$ denotes the $\alpha$-asymptotic volume ratio defined above.
  \label{Sobolev}
 \end{theorem}
 
 \newpage
  
  By setting $f = 1$ in Theorem \ref{Sobolev} we obtain
  the following sharp isoperimetric inequality:
  \begin{corollary} \ \\
  Suppose $(M,g)$ is a smooth complete noncompact Riemannian manifold with positive smooth density $w$ and nonnegative Bakry-Émery Ricci curvature $\Ric_{\alpha}^w$.
  Let $K \subset M$ be a compact subdomain in $M$ with smooth boundary $\partial K$.
  Then we have
   \begin{equation}
    \int_{\partial K} w
    \geq   (m+\alpha) \, \mathcal{V}_{\alpha}^{\frac{1}{m+\alpha}}
    \left( \int_{K} w \right)^{\frac{n+\alpha-1}{n+\alpha}}.
   \end{equation}
  \end{corollary}
  
 In the case of nonnegative Ricci curvature (corresponding to a constant density $w=1$)
 the above inequality was obtained by S.~Brendle
 \cite[Theorem 1.1]{Brendle-NonnegativeCurvature}.
 Previously V.~Agostiniani, M.~Fogagnolo and L.~Mazzieri
 \cite{Agostiniani-Fogagnolo-Mazzieri} 
 proved the inequality in the three-dimensional case by proving a Willmore type inequality.
 Their proof builds on an argument by G.~Huisken \cite{Huisken}.
 Very recently M.~Fogagnolo and L.~Mazzieri \cite{Fogagnolo-Mazzieri} extended the argument up to dimension seven.
  
 Our proof will use the Alexandrov-Bakelman-Pucci (ABP) maximum principle
 as in \cite{Brendle-NonnegativeCurvature}.
 A proof of the isoperimetric inequality in $\R^n$ using this method was first observed by X.~Cabre \cite{Cabre1, Cabre2}, see also work of 
 X.~Cabre, X.~Ros-Oton and J.~Serra
 \cite{Cabre-Ros-Oton-Serra} for various extensions.
 In a recent breakthrough S.~Brendle \cite{Brendle-Submanifold} used this method to prove a sharp Michael--Simon--Sobolev inequality (and hence a sharp isoperimetric inequality) for submanifolds.

 In Section \ref{Proof-FirstMainResult} we discuss the proof of Theorem \ref{Sobolev}.
 In the appendix we provide a direct proof of the Bishop--Gromov volume comparison theorem for smooth manifolds with densities.
  
 \section{Proof of Theorem \ref{Sobolev}} 
 \label{Proof-FirstMainResult}

  We may assume by scaling
  \begin{align*}
   \int_K w |D f| + \int_{\partial K} w f
   =
   (m+\alpha) \int_K w f^{\frac{m+\alpha}{m+\alpha-1}}.
  \end{align*}

Suppose $K$ is a compact domain with smooth boundary $\partial K$. 
We consider the linear Neumann problem given by
\begin{align*}
 \left\{
  \begin{aligned}
   \divA \left( w f D u \right)
   &= (m+\alpha) w f^{\frac{m+\alpha}{m+\alpha -1}} - w |D f| & \text{in} \; K, \\
   \langle Du, \nu \rangle &= 1
   & \text{on} \; \partial K,
  \end{aligned}
 \right.
\end{align*}
  where $\nu$ denotes the outward pointing unit normal vector field of $K$.
  
  The scaling assumption provides the integrability condition for this Neumann problem.
  Since $f$ is smooth, we have $|D f| \in C^{0,1}$ and hence by standard elliptic theory (see for example Theorem 6.31 in \cite{Gilbarg-Trudinger})
  we conclude $u \in C^{2,\gamma}$ for any $0 < \gamma < 1$.

As in \cite{Brendle-NonnegativeCurvature} we define the subset $U \subset K$ by
\begin{align*}
 U 
 =
 \left\{
 x \in K \backslash \partial K | \;
 |D u(x)| < 1
 \right\}.
\end{align*}

For any $r > 0$ we define a subset $A_r$ by
\begin{align*}
 A_r 
 = \left\{ \bar{x} \in U \left| \forall x \in K:
 r u(x) + \frac{1}{2} d \left( x, \exp_{\bar{x}} (r D u(\bar{x})) \right)^2
 \geq r u(\bar{x}) + \frac{1}{2} r^2 |D u(\bar{x})|^2 \right.
 \right\}.
\end{align*}

Let us denote the exponential map at $x$ by $\exp_x: T_x M \rightarrow M$. 
For any $r > 0$ we define the transport map $\Phi_r: K \rightarrow M$ by
\[
 \Phi_r(x) = \exp_x \left( r D u(x) \right).
\]

The transport map is of class $C^{1,\gamma}$ for any $0 < \gamma < 1$ by the above
considerations on the regularity of the solution $u$ of the Neumann problem.

We observe the following inequality (compare with Lemma 2.1 in \cite{Brendle-NonnegativeCurvature}):
\begin{lemma} \ \\
 Assume $x \in U$. Then we have the inequality
 \begin{equation*}
  w \Delta u + \langle D w, D u \rangle
  \leq (m+\alpha) w f^{\frac{1}{m+\alpha-1}}.
 \end{equation*}
 \label{Lemma_EstimateLaplace_u}
\end{lemma}
\begin{proof} \ \\
If we apply
the identity 
 \[
  \divA( w f D u)
  = w f \Delta u + f \langle D w, D u \rangle + w \langle D f, D u \rangle
 \]
to the PDE
 \[
  \divA(wf D u)
  = (m+\alpha) w f^{\frac{m+\alpha}{m+\alpha-1}} - w |D f|,
 \]
we deduce by the Cauchy--Schwarz inequality
\begin{align*}
f ( w \Delta u + \langle D w, D u \rangle)
=
(m+\alpha) w f^{\frac{m+\alpha}{m+\alpha-1}} - w (|D f| + \langle D f, D u \rangle)
\leq
(m+\alpha) w f^{\frac{m+\alpha}{m+\alpha-1}}.
\end{align*}
Dividing by $f > 0$ yields the claim.
\end{proof}

We have the following observation on the image of the transport set:
\begin{lemma}[cf. S.Brendle, Lemma 2.2 in \cite{Brendle-NonnegativeCurvature}] \ \\
 The set
 \[
  \{ p \in M| \, d_g(x,p) < r \; \text{for all} \; x \in K \} 
 \]
is contained in the image $\Phi_r(A_r)$ of the transport map.
\label{Lemma_Inclusion_ABP_Map}
\end{lemma}

\begin{proof} \ \\
 The proof is identical to Lemma 2.2 in \cite{Brendle-NonnegativeCurvature}.
 Hence we omit the details.
\end{proof}

For a point $\bar{x} \in A_r$ we define the path $\bar{\gamma}: [0,r] \rightarrow M$
by $\bar{\gamma}(t) := \exp_x \left( t D u(\bar{x}) \right)$.
We have the following formula for the second variation of energy along $\bar{\gamma}$: 
\begin{lemma}[cf. S.Brendle, Lemma 2.3 in \cite{Brendle-NonnegativeCurvature}] \ \\
Suppose $Z$ is a smooth vector field along $\bar{\gamma}$ vanishing at the end point
(ie. $Z(r) = 0$).
Then we have the second variation formula
\begin{align*}
 (D^2 u)(Z(0), Z(0))
 +
 \int_0^r 
 \left(
  |D_t Z(t)|^2 - \Rm(\bar{\gamma}'(t), Z(t), Z(t), \bar{\gamma}'(t))
 \right)
 \; dt \geq 0.
\end{align*}
\end{lemma}

\begin{proof} \ \\
 The proof is identical to Lemma 2.3 in \cite{Brendle-NonnegativeCurvature}.
 Hence we omit the details.
\end{proof}

The next step is to prove
a vanishing result for Jacobi fields
along the geodesic $\bar{\gamma}$:
\begin{lemma}[cf. S.Brendle, Lemma 2.4 in \cite{Brendle-NonnegativeCurvature}] \ \\
 Choose an orthonormal basis of the tangent space $T_{\bar{x}} M$.
 Suppose $W$ is a Jacobi field along $\bar{\gamma}$ satisfying
 \[
   \langle D_t W(0), e_j \rangle = (D^2 u)(W(0), e_j)
 \]
 for $1 \leq j \leq m$.
 If there exists $\tau \in (0,r)$, such that $W(\tau) = 0$, then $W$ vanishes identically.
\end{lemma}

\begin{proof} \ \\
 The proof is identical to Lemma 2.4 in \cite{Brendle-NonnegativeCurvature}.
 Hence we omit the details.
\end{proof}

The next proposition describes the volume expansion along the transport map $\Phi_t$:
\begin{proposition} \ \\
Assume that $\bar{x} \in A_r$. Then the map
 \[
  t \mapsto
   \left(1+ t f(\bar{x})^{\frac{1}{m+\alpha -1}} \right)
 w\left( \Phi_t(\bar{x}) \right) \; |\det D\Phi_t(\bar{x})| 
 \]
is monotone decreasing for $t \in (0,r)$.
\label{Proposition:VolumeExpansion}
\end{proposition}

\begin{proof} \ \\
Let $\bar{x} \in A_r$, define $\bar{\gamma}: [0,r] \rightarrow M$ by $\bar{\gamma}(t) = \exp_{\bar{x}}(t D u(\bar{x}))$.
Choose an orthonormal basis $\{e_1, \dots, e_m\}$ of the tangent space $T_{\bar{x}} M$,
and construct geodesic normal coordinates $(x^1, \dots, x^m)$ around $\bar{x}$, such that we
have $\partial_i = e_i$ at $\bar{x}$.

We construct for $1 \leq i \leq m$ vector fields $E_i$ along $\bar{\gamma}$ by
parallel transport of the vector fields $e_i$. Moreover, we solve the Jacobi equation
to obtain the unique Jacobi fields $X_i$ along $\bar{\gamma}$ satisfying $X_i(0) = e_i$ and
\[
 \langle D_t X_i (0), e_j \rangle = (D^2 u)(e_i, e_j),
\]
where $D_t$ denotes the covariant derivative along $\bar{\gamma}$.

Let us define a matrix-valued 
function $P: [0,\tau] \rightarrow \Mat(m; \R)$ by
\[
 [P(t)]_{ij}
 =
 \langle X_i(t), E_j(t) \rangle.
\]
We observe by the above properties:
\[
 [P(0)]_{ij} = \delta_{ij} \; \text{and} \; [P'(0)]_{ij} = (D^2 u)(e_i, e_j).
\]

Additionally, we define a matrix-valued function $S: [0,\tau] \rightarrow \Mat(m; \R)$ by 
\[
 [S(t)]_{ij}
 = \Rm \left( \bar{\gamma}'(t), E_i(t), E_j(t), \bar{\gamma}'(t) \right),
\]
where $\Rm$ denotes the Riemann curvature tensor.
For each $t \in [0,\tau]$ the matrix $S(t)$ is symmetric due to the symmetries
of the Riemann curvature tensor.
We have
\[
 \trace S(t) = \Ric\left( \bar{\gamma}'(t), \bar{\gamma}'(t) \right).
\]
By the Jacobi equation for the Jacobi vector fields $X_1, \dots, X_n$ we obtain
\[
 P''(t) = - P(t) S(t).
\]
Moreover, $P'(t) P(t)^T$ is symmetric for each $t \in [0,r]$,
and $P(t)$ is invertible for each $t \in [0,r]$.
We define a matrix-valued function $Q: [0,\tau] \rightarrow \Mat(m; \R)$
by
\[
Q(t)
 = P(t)^{-1} P'(t).
\]
Then $Q(t)$ is symmetric for each $t \in [0,r]$ and it satisfies the Ricci equation
\[
 \frac{d}{dt} Q(t) = - S(t) - Q(t)^2.
\]

We compute
\begin{align*}
 \frac{d}{dt} \trace Q(t)
 &=
 \trace (-S(t) - Q(t)^2) 
 = - \Ric(\bar{\gamma}'(t), \bar{\gamma}'(t))
 - \trace [ Q(t)^2].
\end{align*}

We recall the definition of the Bakry-Émery Ricci curvature:
\[
 \Ric_w^\alpha
 =
 \Ric - D^2 (\log w) - \frac{1}{\alpha} D(\log w) \otimes D (\log w).
\]

Thus
\begin{align*}
 \frac{d}{dt} \trace Q(t)
 =& 
 - (D^2 \log w)(\bar{\gamma}'(t), \bar{\gamma}'(t))
 - \frac{1}{\alpha}
  \langle D (\log w)(\bar{\gamma}(t)), \bar{\gamma}'(t) \rangle^2 
 - \Ric_{\alpha}^w( \bar{\gamma}'(t), \bar{\gamma}'(t))
 -  \trace [Q(t)^2].
\end{align*}

We observe the differential inequality
\begin{align*}
 & \frac{d}{dt} \left[ \trace Q(t) + \langle D \log w(\bar{\gamma}(t)), \bar{\gamma}'(t) \rangle \right]\\
 =& 
 - \frac{1}{\alpha}
  \langle D \log w(\bar{\gamma}(t)), \bar{\gamma}'(t) \rangle^2
 - 
 \Ric_{\alpha}^w (\bar{\gamma}'(t), \bar{\gamma}'(t))
 - \trace [Q(t)^2] \\
 \leq&
 - \frac{1}{m} [\trace Q(t)]^2
 - \frac{1}{\alpha}
  \langle D \log w(\bar{\gamma}(t)), \bar{\gamma}'(t) \rangle^2 \\
 =& 
 - \frac{1}{m+\alpha}
  \left( \trace Q(t) + \langle D \log w(\bar{\gamma}(t)), \bar{\gamma}'(t) \rangle \right)^2
- \frac{m}{\alpha( m+\alpha)}
\left(
 \frac{\alpha}{m} \trace Q(t) -  \langle D \log w(\bar{\gamma}(t)), \gamma'(t) \rangle
\right)^2
  \\
 \leq&
 - \frac{1}{m+\alpha} [ \trace Q(t) + \langle D \log w(\bar{\gamma}(t)), \bar{\gamma}'(t) \rangle]^2,
\end{align*}
where we used our assumption of nonnegative Bakry-Émery Ricci curvature (ie. $\Ric_{\alpha}^w \geq 0$) and the trace inequality for symmetric matrices.

With the help of Lemma \ref{Lemma_EstimateLaplace_u} we observe that the initial value
to the above differential inequality is given by
\[
 \lim_{t \rightarrow 0} [ \trace Q(t) + \langle D \log w(\bar{\gamma}(t)), \bar{\gamma}'(t) \rangle]
 =
 \Delta u(\bar{x}) + \langle D \log w(\bar{x}), \log u(\bar{x}) \rangle
 \leq (m+\alpha)  f(\bar{x})^{\frac{1}{m+\alpha-1}}.
\]

If we apply the comparison principle for ODEs we deduce
the bound
\begin{align*}
   \trace Q(t) + \langle D \log w(\bar{\gamma}(t)), \bar{\gamma}'(t) \rangle \leq \frac{(m+\alpha) f(\bar{x})^{\frac{1}{m+\alpha-1}}}{1 + t f(\bar{x})^{\frac{1}{m+\alpha-1}}}.
\end{align*}

Then we have
\begin{align*}
 \frac{d}{dt} \log \left[  w (\bar{\gamma}(t)) \det P(t) \right] 
 =
   \trace Q(t) + \langle D \log w(\bar{\gamma}(t)), \bar{\gamma}'(t) \rangle
 \leq 
 \frac{(m+\alpha) f^{\frac{1}{m+\alpha-1}}}{1 + t f^{\frac{1}{m+\alpha-1}}}.
\end{align*}

This implies that the function
\[
 t \mapsto
 \left(1+ t f(\bar{x})^{\frac{1}{m+\alpha -1}} \right)^{-(m+\alpha)}
 w(\bar{\gamma}(t)) \det P(t)
\]
is monotone decreasing for $t \in (0,r)$.
The proposition follows by observing that $\det P(t) = |\det D \Phi_t(\bar{x})|$
for any $t \in (0,r)$.
\end{proof}

\begin{corollary} \ \\
 We have for any $x \in A_r$ 
 the relation
 \begin{align*}
  w(\Phi_r(x)) \, |\det D\Phi_r(x)| \leq \left( 1 + r f(x)^{\frac{1}{m+\alpha -1}} \right)^{m+\alpha} w(x).
 \end{align*}
\label{Corollary_Volume_Expansion}
\end{corollary}

We complete the proof of Theorem \ref{Sobolev}: \ \\
Indeed by Lemma \ref{Lemma_Inclusion_ABP_Map}, the change of variables formula and Corollary \ref{Corollary_Volume_Expansion} we have
\begin{align*}
 &\int_{\left\{ p \in M | d_g(x,p) < r \; \text{for all} \; x \in D \right\}} w \; d\mu(x) \\
 \leq&
 \int_{A_r} |\det D\Phi_r(x)| w(\Phi_r(x)) \; d\mu(x) \\
 \leq&
 \int_U \left(1+r f(x)^{\frac{1}{m+\alpha-1}} \right)^{m+\alpha} w(x) \; d\mu(x)
\end{align*}

If we divide by $r^{m+\alpha}$ and send $r \rightarrow \infty$ we obtain
\begin{align*}
\mathcal{V}_{\alpha}
=& 
 \lim_{r \rightarrow \infty} 
 \frac{1}{r^{m+\alpha}}
 \int_{\left\{ p \in M | d_g(x,p) < r \; \text{for all} \; x \in K \right\}} w \; d\mu(x)\\
 \leq& \int_U f(x)^{\frac{m+\alpha}{m+\alpha - 1}} w(x) \; d\mu(x) \\
 \leq& \int_D f(x)^{\frac{m+\alpha}{m+\alpha - 1}} w(x) \; d\mu(x) 
\end{align*}

If we combine the previous estimate with our initial scaling assumption, we deduce
\begin{align*}
 \int_K w |D f|
 + \int_{\partial K} wf
 &= (m+\alpha) \int_{K} wf^{\frac{m+\alpha}{m+\alpha-1}} \\
 &\geq (m+\alpha)
 \left( \int_K wf^{\frac{m+\alpha}{m+\alpha-1}} \right)^{\frac{1}{m+\alpha}}
 \left( \int_K wf^{\frac{m+\alpha}{m+\alpha-1}} \right)^{\frac{m+\alpha-1}{m+\alpha}} \\
 &\geq
 (m+\alpha) \mathcal{V}_{\alpha}^{\frac{1}{m+\alpha}}
  \left( \int_K wf^{\frac{m+\alpha}{m+\alpha-1}} \right)^{\frac{m+\alpha-1}{m+\alpha}}.
\end{align*}
This is the desired estimate.

\appendix

\section{A proof of a Bishop--Gromov volume comparison theorem for smooth manifolds
with densities}

We provide a short proof of the Bishop--Gromov volume comparison theorem
for smooth noncompact manifolds with $\Ric_{\alpha}^w \geq 0$.

\begin{theorem}[K.-T.~Sturm, Theorem 2.3 in \cite{Sturm}] \ \\
 Assume $(M,g)$ is a smooth complete noncompact $m$-dimensional smooth manifold, $\alpha > 0$, and $w$ is a smooth positive function.
 Suppose $(M,g,w)$ has nonnegative Bakry-Émery Ricci curvature with respect to
 the density $w$ and $\alpha > 0$, ie.
 \[
  \Ric_{\alpha}^w = \Ric - D^2 (\log w) - \frac{1}{\alpha} D \log w \otimes D \log w \geq 0.
 \]
 Then the function 
 \begin{align*}
  r \mapsto 
  \frac{1}{r^{m+\alpha}} \int_{B_r (q)} w, \; \text{where} \; B_r(q) 
  = \left\{ p \in M | d_g(p,q) < r \right\}
 \end{align*}
 is monotone decreasing for any $q \in M$.
\end{theorem}

\begin{proof} \ \\
 We define $\Sigma_t \subset M$ by 
 \[
  \Sigma_t = \left\{ \exp_q (v)| \; v \in \seg^0(q) \; \text{and} \; |v| = t \right\}.
 \]
 The set $\Sigma_t$ is the distance set avoiding the cut locus, here $\seg^0(p)$
 denotes the interior segment domain. 
 We denote the second fundamental form of the hypersurface $\Sigma_t$ by $h$ and
 its trace, the mean curvature of the hypersurface $\Sigma_t$, by $H$.
 
 Let $v \in T_q M$, $|v| = 1$ and consider the radial geodesic $\gamma(t) = \exp_q (tv)$ such that $\gamma(t) \in \Sigma_t$.
 Then we have the Jacobi equation
 \[
  \left(\frac{d}{dt} H \right)(\gamma(t))
  = - |h|^2(\gamma(t)) - \Ric(\gamma'(t), \gamma'(t))
 \]
 provided $tv \in \seg^0(p)$. 
 This implies
 \begin{align*}
  &\frac{d}{dt} \left[ H(\gamma(t)) + \langle \gamma'(t), D \log w(\gamma(t)) \rangle \right] \\
  =&
  \left(\frac{d}{dt} H \right) (\gamma(t)) + (D^2 \log w)\left(\gamma'(t), \gamma'(t) \right) \\
  =&
  - |h|^2(\gamma(t)) - \Ric(\gamma'(t), \gamma'(t))
  + (D^2 \log w)\left(\gamma'(t), \gamma'(t) \right) \\
  =& - |h|^2(\gamma(t)) - \frac{1}{\alpha} \langle D \log w(\gamma'(t)), \gamma'(t) \rangle - \Ric_{\alpha}^w(\gamma'(t), \gamma'(t)) \\
  \leq&
  - \frac{1}{m-1} H^2(\gamma(t)) - \frac{1}{\alpha} \langle D \log w(\gamma'(t)), \gamma'(t) \rangle \\
  =&
     - \frac{1}{m-1+\alpha} \left[ H(\gamma(t)) + \langle \gamma'(t), D \log w(\gamma(t)) \rangle \right]^2 \\
     &- \frac{m-1}{\alpha(m-1+\alpha)}
     \left[
      \frac{\alpha}{m-1} H(\gamma(t)) - \langle \gamma'(t), D \log w(\gamma(t)) \rangle
     \right]^2
  \\
  \leq& 
   - \frac{1}{m-1+\alpha} \left[ H(\gamma(t)) + \langle \gamma'(t), D \log w(\gamma(t)) \rangle \right]^2,
\end{align*}
where we have used the trace inequality for the second fundamental form
and the nonnegativity of the Bakry--Émery Ricci curvature.

Thus we established the differential inequality
\begin{align*}
 \frac{d}{dt} \left[ H(\gamma(t)) + \langle \gamma'(t), D \log w(\gamma(t)) \rangle \right]
 \leq
 - \frac{1}{m-1+\alpha} \left[ H(\gamma(t)) + \langle \gamma'(t), D \log w(\gamma(t)) \rangle \right]^2.
\end{align*}

 Integrating the above differential inequality implies
 \begin{align*}
  H(\gamma(t)) + \langle \gamma'(t), D \log w(\gamma(t)) \rangle
  \leq \frac{m-1+\alpha}{t}.
 \end{align*}

 This implies
 \begin{align*}
  \frac{d}{dt}
  \left( \int_{\Sigma_t} w \right)
  \leq
  \int_{\Sigma_t} \left( H + \langle \nu, D \log w \rangle \right) w
  \leq
  \frac{m-1+\alpha}{t} \int_{\Sigma_t} w,
 \end{align*}
 where the first inequality holds because the interior segment domain $\seg^0(p)$ is star-shaped and the second inequality is our differential inequality.

 Integration of this differential inequality implies that the map
 \[
  t \mapsto t^{-(m-1+\alpha)} \int_{\Sigma_t} w
 \]
 is decreasing.
 We observe for any $r > 0$ by the coarea formula
 \begin{align*}
  r^{-(m+\alpha)} \int_{B_r(q)} w
  =
  r^{-(m+\alpha)} \int_0^r \left( \int_{\Sigma_t} w \right) \; dt
  =
  \int_0^1 \left( r^{-(m-1+\alpha)} \int_{\Sigma_{\tau r}} w \right) \; d \tau,
 \end{align*}
 where we use that the cut locus is a set of measure zero.
 The inner bracket on the right-hand side is nonincreasing in the radius $r$ by the previous observation,
 hence the left-hand side is nonincreasing in the radius $r$ as an average of nonincreasing functions.
 
\end{proof}

\end{document}